\documentclass[12pt,reqno]{amsart}
\usepackage{amsmath, amsthm, amssymb}

\advance \topmargin by -\headheight
\advance \topmargin by -\headsep
     
\evensidemargin \oddsidemargin
\newtheorem{theorem}{Theorem}[section]
\newtheorem{lemma}[theorem]{Lemma}
\newtheorem{corollary}[theorem]{Corollary}

\theoremstyle{definition}

\theoremstyle{remark}

\numberwithin{equation}{section}

\newcommand{\mmod}[1]{\,\,(\text{mod}\,\,#1)}

\def\bfi{{\mathbf i}}
 
\def\bfm{{\mathbf m}}

\def\calA{{\mathcal A}}

\def\calI{{\mathcal I}}

\def\calM{{\mathcal M}}

\def\dbN{{\mathbb N}}
\def\dbR{{\mathbb R}}
\def\dbZ{{\mathbb Z}}

\def\grA{{\mathfrak A}}
\def\grB{{\mathfrak B}}
\def\grC{{\mathfrak C}}

\def\grB{{\mathfrak B}}\def\grC{{\mathfrak C}}

\def\alp{{\alpha}} \def\bfalp{{\boldsymbol \alpha}}
\def\bet{{\beta}}  \def\bfbet{{\boldsymbol \beta}}
\def\gam{{\gamma}}  \def\bfgam{{\boldsymbol \gamma}}
\def\del{{\delta}}

\def\tet{{\theta}}  
\def\bfiota{{\boldsymbol \iota}}

\def\lam{{\lambda}}

\def\sig{{\sigma}}

\def\d{{\partial}}
\def\eps{\varepsilon}

\def\le{\leqslant} \def\ge{\geqslant}

\def\d{{\,{\rm d}}}

\begin{document}
\title[Perturbations of Weyl sums]{Perturbations of Weyl sums}
\author[Trevor D. Wooley]{Trevor D. Wooley}
\address{School of Mathematics, University of Bristol, University Walk, Clifton, Bristol BS8 
1TW, United 
Kingdom}
\email{matdw@bristol.ac.uk}
\date{}
\begin{abstract} Write $f_k(\bfalp;X)=\sum_{x\le X}e(\alp_1x+\ldots +\alp_kx^k)$ 
$(k\ge 3)$. We show that there is a set $\grB\subseteq [0,1)^{k-2}$ of full measure with 
the property that whenever $(\alp_2,\ldots ,\alp_{k-1})\in \grB$ and $X$ is sufficiently 
large, then
$$\sup_{(\alp_1,\alp_k)\in [0,1)^2}|f_k(\bfalp;X)|\le X^{1/2+4/(2k-1)}.$$
For $k\ge 5$, this improves on work of Flaminio and Forni, in which a Diophantine condition is 
imposed on $\alp_k$, and the exponent of $X$ is $1-2/(3k(k-1))$.
\end{abstract}
\maketitle

\section{Introduction} Consider the exponential sum $f_k(\bfalp;X)$, defined for $k\ge 2$ 
and $\bfalp\in \dbR^k$ by
\begin{equation}\label{1.1}
f_k(\bfalp;X)=\sum_{1\le x\le X}e(\alp_1x+\ldots +\alp_kx^k),
\end{equation}
where, as usual, we write $e(z)=e^{2\pi iz}$. It was shown by H. Weyl \cite{Wey1916} 
that when $\alp_k$ is irrational, then $\lim\sup\, X^{-1}|f_k(\bfalp;X)|=0$ as 
$X\rightarrow \infty$. Indeed, when $\alp_k$ satisfies an appropriate Diophantine 
condition, as is the case for algebraic irrational numbers such as $\sqrt{2}$, then for each 
$\eps>0$, provided only that $X$ is sufficiently large in terms of $k$ and $\eps$, one has 
the upper bound
\begin{equation}\label{1.2}
|f_k(\bfalp;X)|\le X^{1-2^{1-k}+\eps}.
\end{equation}
Although such conclusions can be improved by employing the latest developments 
surrounding Vinogradov's mean value theorem (see, for example 
\cite[Theorem 1.5]{Woo2012}), the improved exponents remain very close to $1$. 
Motivated by recent work of Flaminio and Forni \cite{FF2014} concerning equidistribution 
for higher step nilflows, in this paper we address two basic questions. First, we explore the 
extent to which the estimate (\ref{1.1}) can be improved if one is prepared to exclude the 
perturbing coefficient tuple $(\alp_1,\ldots ,\alp_{k-1})$ from a set of measure zero. 
Second, we examine how sensitive such estimates may be to the Diophantine conditions 
imposed on the lead coefficient $\alp_k$.\par

Before proceeding further, we introduce some notation associated with Vinogradov's mean 
value theorem. With $f_k(\bfalp;X)$ defined via (\ref{1.1}), the {\it Main Conjecture} 
asserts that for all positive numbers $s$, one has
\begin{equation}\label{1.3}
\int_{[0,1)^k}|f_k(\bfalp;X)|^{2s}\d\bfalp \ll X^\eps(X^s+X^{2s-\frac{1}{2}k(k+1)}).
\end{equation}
Here and throughout, the implicit constant in Vinogradov's notation may depend on $k$, 
$s$ and the arbitrary positive number $\eps$. We denote by ${\rm MC}_k(u)$ the 
assertion that the Main Conjecture (\ref{1.3}) holds for $1\le s\le u$. We will be interested 
in the size of the exponential sum $f_k(\bfalp;X)$ when the coefficients $\alp_i$ are fixed 
for certain suffices $i=i_l$ $(1\le l\le t)$ with $1\le i_1<i_2<\ldots <i_t\le k$. The 
complementary set of suffices
$$\{1,2,\ldots ,k\}\setminus \{i_1,i_2,\ldots ,i_t\}=\{ \iota_1,\iota_2,\ldots \iota_{k-t}\},
$$
with $1\le \iota_1<\iota_2<\ldots <\iota_{k-t}\le k$, then corresponds to a $(k-t)$-tuple 
$(\alp_{\iota_1},\ldots ,\alp_{\iota_{k-t}})$ that we permit to come from a set 
$\grB(\bfiota)\subseteq [0,1)^{k-t}$ that is central to our investigations. In order to 
facilitate concision, throughout this paper we write $\bfalp^*$ for 
$(\alp_{i_1},\ldots ,\alp_{i_t})$ and $\bfalp^\dagger$ for 
$(\alp_{\iota_1},\ldots ,\alp_{\iota_{k-t}})$.\par

\begin{theorem}\label{theorem1.1} Suppose that $k\ge 3$ and 
$1\le u\le \tfrac{1}{2}k(k+1)$, and assume ${\rm MC}_k(u)$. Let $t$ be a positive 
integer with $1\le t\le k$, and let $\bfi$ be a $t$-tuple of suffices satisfying 
$1\le i_1<i_2<\ldots <i_t\le k$. Then there exists a set 
$\grB(\bfiota)\subseteq [0,1)^{k-t}$ of full measure such that, whenever 
$(\alp_{\iota_1},\ldots ,\alp_{\iota_{k-t}})\in \grB(\bfiota)$, then for all real numbers $X$ 
sufficiently large in terms of $\eps$, $k$ and $\bfalp^\dagger$, one has
$$\sup_{\bfalp^*\in [0,1)^t}|f_k(\bfalp;X)|\le X^{1/2+\del(\bfi)+\eps},$$
where
\begin{equation}\label{1.4}
\del(\bfi)=\frac{t+1+2(i_1+\ldots +i_t)}{4u+2t+2}.
\end{equation}
\end{theorem}

\begin{corollary}\label{corollary1.2} Suppose that $k\ge 3$. Then there exists a set 
$\grB\subseteq [0,1)^{k-2}$ of full measure such that, whenever 
$(\alp_2,\alp_3,\ldots ,\alp_{k-1})\in \grB$, then for all real numbers $X$ sufficiently large 
in terms of $k$ and $\alp_2,\ldots ,\alp_{k-1}$, one has
\begin{equation}\label{1.5}
\sup_{(\alp_1,\alp_k)\in [0,1)^2}|f_k(\bfalp;X)|\le X^{1/2+\del_k},
\end{equation}
where $\del_k=4/(2k-1)$. Moreover, when $k$ is sufficiently large, the same conclusion 
holds with $\del_k=1/k+o(1)$.
\end{corollary}

When $k$ is large, the conclusion of Corollary \ref{corollary1.2} obtains very nearly 
square-root cancellation for the exponential sum $f_k(\bfalp;X)$, greatly improving the 
estimate (\ref{1.2}). In addition to this emphatic response to the first question posed in 
our opening paragraph, we note that no condition whatsoever has been imposed on the 
lead coefficient $\alp_k$. Of course, the restriction of the $(k-2)$-tuple 
$(\alp_2,\ldots ,\alp_{k-1})$ to the universal set $\grB$ of measure $1$ implicitly imposes 
some sort of Diophantine condition on these lower order coefficients. Nonetheless, it is 
clear that there is in general little sensitivity to the lead coefficient.\par

Flaminio and Forni \cite[Corollary 1.2]{FF2014} have derived a conclusion similar to that 
of Corollary \ref{corollary1.2} in which $\alp_k$ is subject to a certain Diophantine 
condition, and the conclusion (\ref{1.5}) holds with 
$\tfrac{1}{2}+\del_k=1-1/\left(\frac{3}{2}k(k-1)\right)$. Subject to a similar Diophantine 
condition on $\alp_k$, the latest progress on Vinogradov's mean value theorem permits 
the proof of a similar estimate with $2(k-1)(k-2)$ in place of $\frac{3}{2}k(k-1)$, though 
without any restriction on $(\alp_2,\ldots ,\alp_{k-1})$ (simply substitute the conclusion of 
\cite[Theorem 1.2]{Woo2014c} into the argument of the proof of 
\cite[Theorem 11.1]{Woo2014a}). Thus, when $k$ is large, the conclusion of Flaminio and 
Forni obtains barely non-trivial cancellation subject to a Diophantine condition, whereas 
Corollary \ref{corollary1.2} delivers nearly square-root cancellation.\par

We have aligned Corollary \ref{corollary1.2} so as to facilitate comparison with the work of 
Flaminio and Forni \cite[Corollary 1.2]{FF2014}. When $k$ is large, the conclusion of 
Theorem \ref{theorem1.1} offers estimates for $f_k(\bfalp;X)$ exhibiting close to 
square-root cancellation even when the number of fixed coefficients $\alp_i$ is large. We 
illustrate such ideas with a further corollary.

\begin{corollary}\label{corollary1.3} Suppose that $k$ is large, and that $i_l$ 
$(1\le l\le t)$ are integers with $1\le i_1<i_2<\ldots <i_t\le k$. Suppose also that 
$i_1+\ldots +i_t+t+1<\tfrac{1}{2}k^2/(\log k)$. Then there exists a set 
$\grB(\bfiota)\subseteq [0,1)^{k-t}$ of full measure such that, whenever 
$\bfalp^\dagger \in \grB(\bfiota)$, then for all $X$ sufficiently large in terms of $k$ and 
$\bfalp^\dagger$, one has
$$\sup_{\bfalp^*\in [0,1)^t}|f_k(\bfalp;X)|\le X^{1/2+1/\log k}.$$
\end{corollary}

The conclusion of Corollary \ref{corollary1.3} shows that, in a suitable sense, almost a 
positive proportion of the coefficients of $f_k(\bfalp;X)$ can be fixed, and yet one 
nonetheless achieves nearly square-root cancellation on a universal set of full measure for 
the remaining coefficients.\par

Our methods extend naturally to deliver equidistribution results for polynomials modulo 
$1$. In this context, when $0\le a<b\le 1$, we write $Z_{a,b}(\bfalp;N)$ for the number 
of integers $n$ with $1\le n\le N$ for which
$$a\le \alp_1n+\alp_2n^2+\ldots +\alp_kn^k\le b\mmod{1}.$$

\begin{theorem}\label{theorem1.4} Suppose that $k\ge 3$ and 
$1\le u\le \tfrac{1}{2}k(k+1)$, and assume ${\rm MC}_k(u)$. Let $t$ be a positive 
integer with $1\le t\le k$, and let $\bfi$ be a $t$-tuple of suffices satisfying 
$1\le i_1<i_2<\ldots <i_t\le k$. Then there exists a set 
$\grB(\bfiota)\subseteq [0,1)^{k-t}$ of full measure such that, whenever 
$(\alp_{\iota_1},\ldots ,\alp_{\iota_{k-t}})\in \grB(\bfiota)$, then for all real numbers $N$ 
sufficiently large in terms of $\eps$, $k$ and $\bfalp^\dagger$, one has
$$|Z_{a,b}(\bfalp;N)-(b-a)N|\le N^{1/2+\nu(\bfi)+\eps}\quad (0\le a<b\le 1),$$
where
\begin{equation}\label{1.6}
\nu(\bfi)=\frac{t+2+2(i_1+\ldots +i_t)}{4u+2t+4}.
\end{equation}
\end{theorem}

\begin{corollary}\label{corollary1.5} Suppose that $k\ge 3$. Then there exists a set 
$\grB\subseteq [0,1)^{k-1}$ of full measure such that, whenever 
$(\alp_1,\alp_2,\ldots ,\alp_{k-1})\in \grB$, then for all real numbers $N$ sufficiently large 
in terms of $k$ and $\alp_1,\ldots ,\alp_{k-1}$, one has
$$|Z_{a,b}(\bfalp;N)-(b-a)N|\le N^{1/2+2/k}\quad (0\le a<b\le 1).$$
\end{corollary}

Write $\|\tet\|=\min\{|\tet-m|:m\in \dbZ\}$. Then by putting $a=0$ and 
$b=N^{-1/2+2/k}$, we obtain as a special case of Corollary \ref{corollary1.5} the 
following conclusion.

\begin{corollary}\label{corollary1.6} Suppose that $k\ge 3$. Then there exists a set 
$\grB^*\subseteq [0,1)^{k-1}$ of full measure such that, whenever 
$(\alp_1,\ldots ,\alp_{k-1})\in \grB^*$, then for all real numbers $N$ sufficiently large in 
terms of $k$ and $\alp_1,\ldots ,\alp_{k-1}$, one has
\begin{equation}\label{1.7}
\min_{1\le n\le N}\|\alp_1n+\ldots +\alp_kn^k\| \ll N^{-1/2+2/k}.
\end{equation}
\end{corollary}

There are results available in the literature analogous to (\ref{1.7}) in which 
$(\alp_1,\ldots ,\alp_{k-1})$ is a fixed real $(k-1)$-tuple. Thus one finds that the 
conclusions of \cite[Theorem 5.2]{Bak1986} and \cite[Theorem 11.3]{Woo2014a} (as 
enhanced by utilising \cite[Theorem 1.2]{Woo2014c}) yield an estimate of the shape 
(\ref{1.7}) with the exponent $\tfrac{1}{2}+\tfrac{2}{k}$ replaced by any real number 
exceeding $1-1/\min\{4(k-1)(k-2),2^{k-1}\}$. These uniform results are considerably 
weaker than those available via Corollary \ref{corollary1.6}.\par

In contrast to the ergodic methods employed by Flaminio and Forni \cite{FF2014}, in this 
paper we utilise recent progress on Vinogradov's mean value theorem. Of critical 
importance to us are mean value estimates of the shape
\begin{equation}\label{1.8}
\int_{[0,1)^k}|f_k(\bfalp;X)|^{2s}\d\bfalp \ll X^{s+\del},
\end{equation}
with $\del$ small and $s$ large. Prior to the author's introduction of ``efficient 
congruencing'' methods in 2012 (see \cite{Woo2012}), available estimates were far too 
weak to deliver conclusions of the type described in Corollary \ref{corollary1.2}. However, 
the estimate (\ref{1.8}) is established in \cite[Corollary 1.3]{Woo2013} with 
$\del=1+\eps$ for $1\le s\le \tfrac{1}{4}k^2+k$, and this would suffice for our purposes 
in the present paper. Recent work of Ford \cite[Theorem 1.1]{FW2014} joint with the 
author establishes (\ref{1.8}) for any $\del>0$ in the same range of $s$, and even more 
recently the author \cite[Theorem 1.3]{Woo2014b} has extended the permissible range of 
$s$ to $1\le s\le \tfrac{1}{2}k(k+1)-\tfrac{1}{3}k+o(k)$, encompassing nearly the whole 
of the critical interval.\par

Let $X$ and $T$ be large, and consider a fixed $t$-tuple 
$(\alp_{i_1},\ldots ,\alp_{i_t})\in [0,1)^t$. The estimate (\ref{1.8}) permits one to 
estimate the measure of the set $\grB_T(X)$ of $(k-t)$-tuples $(\alp_{\iota_1},\ldots 
\alp_{\iota_{k-t}})\in [0,1)^{k-t}$ for which $|f_k(\bfalp;X)|>T$. Suppose that $T$ is 
chosen as a function of $X$ for which $\sum_{X=1}^\infty\text{mes}(\grB_T(X))<\infty$, 
and define $\grB^*\subseteq [0,1)^{k-t}$ to be the set of $(k-t)$-tuples 
$(\alp_{\iota_1},\ldots ,\alp_{\iota_{k-t}})$ for which 
$\limsup T^{-1}|f_k(\bfalp;X)|\ge 1$ as $X\rightarrow \infty$. Then it follows from the 
Borel-Cantelli theorem that the set $\grB^*$ has measure $0$. One may remove the 
dependence of these estimates on the fixed $t$-tuple of coefficients 
$(\alp_{i_1},\ldots ,\alp_{i_t})$ by a suitable application of the mean value theorem, 
showing that the size of $|f_k(\bfalp;X)|$ changes little as $\alp_j$ varies over an interval 
having length of order $X^{-j}$. Moreover, we are able to sharpen our estimates by 
observing that $|f_k(\bfalp;X)|$ also changes little as $X$ varies over an interval of length 
small compared to $T$.\par

We remark that Pustyl$^\prime$nikov has work spanning a number of papers (see, 
for example \cite{Pus1991}) which derives conclusions related to those of this paper. 
Pustyl$^\prime$nikov makes use of the estimate (\ref{1.8}) in the classical case 
$s=k$. In this special case, one may apply Newton's formulae relating symmetric 
polynomials with the roots of polynomials to derive the formula
$$\int_{[0,1)^k}|f_k(\bfalp;X)|^{2s}\d\bfalp \sim s!X^s.$$
The point of view taken in \cite{Pus1991} is that by taking $k$ sufficiently large, one may 
gain some control of the value distribution of Weyl sums $f_k(\bfalp;X)$. The relative 
strength of the conclusions made available in the present paper rests on the far more 
powerful mean value estimates stemming from our recent work on Vinogradov's mean 
value theorem.\par

Our basic parameter is $X$, a sufficiently large positive number. In this paper, implicit 
constants in Vinogradov's notation $\ll$ and $\gg$ may depend on $k$, $u$ and $\eps$. 
Whenever $\eps$ appears in a statement, either implicitly or explicitly, we assert that the 
statement holds for each $\eps>0$. We use vector notation in the natural way. When 
$\grA\subset \dbR$ is Lebesgue measurable, we write $\mu(\grA)$ for its measure. 
Finally, we write $[\tet]$ for $\max\{n\in \dbZ:n\le \tet\}$.
  
The author is grateful to Professors Flaminio and Forni for discussions concerning the 
problems addressed in this paper, and in particular for providing the author with an early 
version of their paper \cite{FF2014}. These discussions benefitted from the excellent 
working conditions and support provided by the Isaac Newton Institute in Cambridge  
during the program ``Interactions between Dynamics of Group Actions and Number 
Theory'' in June 2014.

\section{Large values of Weyl sums} Our goal in this section is the proof of Theorem 
\ref{theorem1.1} and its corollaries. We begin our analysis of $f_k(\bfalp;X)$ by showing 
that the magnitude of this Weyl sum changes little when its argument is modified by a 
small quantity.

\begin{lemma}\label{lemma2.1} Let $T>0$ and $\bfalp\in \dbR^k$, and suppose that 
$|f_k(\bfalp;X)|>T$. Then whenever $\bfbet\in \dbR^k$ satisfies 
$$|\bet_j-\alp_j|\le (4\pi k)^{-1}TX^{-j-1}\quad (1\le j\le k),$$
one has $|f_k(\bfbet;X)|>\tfrac{1}{2}T$.
\end{lemma}

\begin{proof} Under the hypotheses of the statement of the lemma, an application of the 
multidimensional mean value theorem (see \cite[Theorem 6-17]{Apo1957}) shows that 
there exists a point $\bfgam$ on the line segment connecting $\bfalp$ and $\bfbet$ such 
that
\begin{align*}
f_k(\bfbet;X)-f_k(\bfalp;X)&=\sum_{j=1}^k(\bet_j-\alp_j)\frac{\partial}{\partial \gam_j}
f_k(\bfgam;X)\\
&=2\pi i\sum_{j=1}^k(\bet_j-\alp_j)\sum_{1\le x\le X}x^je(\gam_1x+\ldots +\gam_kx^k).
\end{align*}
Thus, by making a trivial estimate for the exponential sum defined by the inner summation 
here, we deduce that
\begin{align*}
|f_k(\bfbet;X)|&\ge |f_k(\bfalp;X)|-2\pi \sum_{j=1}^k|\bet_j-\alp_j|X^{j+1}\\
&>T-(2k)^{-1}\sum_{j=1}^kT=\tfrac{1}{2}T.
\end{align*}
This completes the proof of the lemma.
\end{proof}

We suppose now that $i_l$ $(1\le l\le t)$ are suffices with $1\le i_1<\ldots <i_t\le k$, 
and we recall the notation introduced in the preamble to the statement of Theorem 
\ref{theorem1.1} above. It is convenient to write
$$\sig(\bfi)=i_1+i_2+\ldots +i_t.$$
Our initial objective is to obtain an estimate for the set
\begin{equation}\label{2.1}
\grB_T(X)=\{ \bfalp^\dagger\in [0,1)^{k-t}:\text{$|f_k(\bfalp;X)|>T$ for some 
$\bfalp^*\in [0,1)^t$}\}.
\end{equation}

\begin{lemma}\label{lemma2.2} Suppose that $1\le u\le \tfrac{1}{2}k(k+1)$, and assume 
the hypothesis ${\rm MC}_k(u)$. Then whenever $T$ is a real number with $0<T\le X$, 
one has
$$\mu(\grB_T(X))\ll X^{u+t+\sig(\bfi)+\eps}T^{-2u-t}.$$
\end{lemma}

\begin{proof} For $1\le l\le t$, put
$$\del_l=(4\pi k)^{-1}TX^{-i_l-1}\quad \text{and}\quad M_l=[\del_l^{-1}].$$
When $0\le m_l\le M_l$ $(1\le l\le t)$, we define the hypercuboids
$$\calI(\bfm)=[m_1\del_1,(m_1+1)\del_1]\times \ldots \times [m_t\del_t,(m_t+1)\del_t]$$
and
$$\calM=[0,M_1]\times \ldots \times [0,M_t].$$
Finally, for each $\bfm\in \calM$, we put
$$\grB_T(\bfm;X)=\{ \bfalp^\dagger\in [0,1)^{k-t}:\text{$|f_k(\bfalp;X)|>T$ for some 
$\bfalp^*\in \calI(\bfm)$}\}.$$
Since $[0,1)^t$ is contained in the union of the sets $\calI(\bfm)$ for $\bfm\in \calM$, we 
see that
\begin{equation}\label{2.2}
\grB_T(X)=\bigcup_{\bfm\in \calM}\grB_T(\bfm;X).
\end{equation}

\par Observe next that when $\bfalp^*$ and $\bfbet^*$ both lie in $\calI(\bfm)$ for 
some $\bfm\in \calM$, then
$$|\alp_{i_l}-\bet_{i_l}|\le \del_l=(4\pi k)^{-1}TX^{-i_l-1}\quad (1\le l\le t).$$
Thus we deduce from Lemma \ref{lemma2.1} that whenever 
$\bfalp^\dagger\in \grB_T(\bfm;X)$ for some $\bfm\in \calM$, then 
$|f_k(\bfalp;X)|>\tfrac{1}{2}T$ for all $\bfalp^*\in \calI(\bfm)$. It follows that
$$\left(\tfrac{1}{2}T\right)^{2u}\mu(\grB_T(\bfm;X))
\mu(\calI(\bfm))<\int_{\calI(\bfm)}\int_{\grB_T(\bfm;X)}
|f_k(\bfalp;X)|^{2u}\d\bfalp^\dagger \d\bfalp^*.$$
But $\mu(\calI(\bfm))=\del_1\cdots \del_t\gg (T/X)^tX^{-\sig(\bfi)}$, and thus
$$T^{2u+t}X^{-t-\sig(\bfi)}\mu(\grB_T(\bfm;X))\ll \int_{\calI(\bfm)}
\int_{[0,1)^{k-t}}|f_k(\bfalp;X)|^{2u}\d\bfalp^\dagger \d\bfalp^*.$$
Consequently, on recalling (\ref{2.2}), one arrives at the upper bound
\begin{align*}
\mu(\grB_T(X))&\le \sum_{\bfm\in \calM}\mu(\grB_T(\bfm;X))\\
&\ll T^{-2u-t}X^{t+\sig(\bfi)}\sum_{\bfm\in \calM}\int_{\calI(\bfm)}\int_{[0,1)^{k-t}}
|f_k(\bfalp;X)|^{2u}\d\bfalp^\dagger \d\bfalp^*.
\end{align*}
Since the union of the sets $\calI(\bfm)$ with $\bfm\in \calM$ is contained in $[0,2)^t$, we 
reach the point at which we may utilise ${\rm MC}_k(u)$, obtaining the estimate
\begin{align*}
\mu(\grB_T(X))&\ll T^{-2u-t}X^{t+\sig(\bfi)}\int_{[0,2)^k}|f_k(\bfalp;X)|^{2u}
\d\bfalp \\
&\ll T^{-2u-t}X^{t+\sig(\bfi)}\cdot 2^kX^{u+\eps}.
\end{align*}
The conclusion of the lemma is now immediate.
\end{proof}

We next make a choice for $T$. Let $\tau$ be a positive number, and put
$$T(X)=X^{1/2+\del(\bfi)+\tau},$$
where $\del(\bfi)$ is defined as in (\ref{1.4}). Here we note that
\begin{equation}\label{2.3}
\tfrac{1}{2}+\del(\bfi)=\frac{(2u+t+1)+(t+1+2\sig(\bfi))}{4u+2t+2}
=\frac{u+t+1+\sig(\bfi)}{2u+t+1}.
\end{equation}
Finally, let $(X_n)_{n=1}^\infty$ be any sequence of natural numbers with the property 
that for large enough values of $n$, one has
\begin{equation}\label{2.4}
T(X_n)\le X_{n+1}-X_n\le 2T(X_n),
\end{equation}
and in the interests of concision, write $T_n=T(X_n)$.

\begin{lemma}\label{lemma2.3} Suppose that $1\le u\le \tfrac{1}{2}k(k+1)$ and assume 
${\rm MC}_k(u)$. Then for any sequence $(X_n)_{n=1}^\infty$ satisfying 
{\rm (\ref{2.4})}, one has
$$\mu\left( \bigcup_{n=1}^\infty \grB_{T_n}(X_n)\right) <\infty .$$
\end{lemma}

\begin{proof} On noting the relation (\ref{2.3}), we find from Lemma \ref{lemma2.2} that
$$\sum_{n=1}^\infty \mu(\grB_{T_n}(X_n))\ll \sum_{n=1}^\infty 
X_n^{u+t+\sig(\bfi)+\eps}T_n^{-2u-t}\le \sum_{n=1}^\infty (T_n/X_n)X_n^{\eps-2\tau}.
$$
In view of the condition (\ref{2.4}), it follows that whenever $m\ge X_nT_n^{-1}$, then 
$X_{n+m}-X_n\ge (X_n/T_n)T_n$, whence $X_{n+m}\ge 2X_n$. Consequently,
$$\mu\left( \bigcup_{n=1}^\infty \grB_{T_n}(X_n)\right) \ll \sum_{j=0}^\infty 
(2^j)^{\eps-2\tau}<\infty .$$
This completes the proof of the lemma.
\end{proof}

We are now equipped to complete the proof of Theorem \ref{theorem1.1}. Denote by 
$A_n(\bfalp^\dagger)$ the condition that $|f_k(\bfalp;X_n)|>T_n$ for some 
$\bfalp^*\in [0,1)^t$. Then the definition (\ref{2.1}) of $\grB_T(X)$ implies that
$$\grB_{T_n}(X_n)=\{ \bfalp^\dagger \in [0,1)^{k-t}:A_n(\bfalp^\dagger)\}.$$
Put
$$\grB^*=\{ \bfalp^\dagger\in [0,1)^{k-t}:\text{$A_n(\bfalp^\dagger)$ holds for infinitely 
many $n\in \dbN$}\}.$$
Then it follows from Lemma \ref{lemma2.3} via the Borel-Cantelli lemma that 
$\mu(\grB^*)=0$. Consequently, there is a set $\grB_0=[0,1)^{k-t}\setminus \grB^*$ of 
full measure having the property that, whenever $\bfalp^\dagger\in \grB_0$, then 
$A_n(\bfalp^\dagger)$ holds for at most finitely many $n\in \dbN$. The latter assertion 
implies that $|f_k(\bfalp;X_n)|\le T_n$ for all $\bfalp^*\in [0,1)^t$, with the exception of 
at most finitely many $n\in \dbN$.\par

Suppose that $X>0$, and put $X^*=[X]$, so that $f_k(\bfalp;X)=f_k(\bfalp;X^*)$. In view 
of the condition (\ref{2.4}), when $X$ is sufficiently large there exists $n\in \dbN$ for 
which $X_n\le X^*\le X_n+2T_n$. But then, on making a trivial estimate for the 
exponential function, we have
$$|f_k(\bfalp;X)-f_k(\bfalp;X_n)|\le X-X_n\le 2T_n.$$
Whenever $|f_k(\bfalp;X_n)|\le T_n$, therefore, one finds that
$$|f_k(\bfalp;X)|\le 3T_n=3X_n^{1/2+\del(\bfi)+\tau}\le 3X^{1/2+\del(\bfi)+\tau}.$$
Then we may conclude that whenever $\bfalp^\dagger \in \grB_0$, then for all positive 
numbers $X$, one has $|f_k(\bfalp;X)|\le 3X^{1/2+\del(\bfi)+\tau}$ for all 
$\bfalp^*\in [0,1)^t$, with the exception of at most those numbers $X$ lying in a 
bounded interval $(0,X_0]$. Since $\tau>0$ may be taken arbitrarily small, the conclusion 
of Theorem \ref{theorem1.1} follows.\vskip.1cm

The corollaries to Theorem \ref{theorem1.1} are easily confirmed. On the one hand, when 
$k\ge 4$, we find from \cite[Theorem 1.1]{FW2014} that ${\rm MC}_k(u)$ holds for 
$u=\left[ \tfrac{1}{4}(k+1)^2\right]$. On the other hand, from 
\cite[Theorem 1.3]{Woo2014b}, one obtains ${\rm MC}_k(u)$ when $k$ is large and 
$u=[\tfrac{1}{2}k(k+1)-\tfrac{1}{3}k-8k^{2/3}]$. In order to establish Corollary 
\ref{corollary1.2}, we apply Theorem \ref{theorem1.1} with $\bfi=(1,k)$. In such 
circumstances, we have $t=2$ and
$$\del(\bfi)=\frac{3+2(k+1)}{4u+6}.$$
Thus, when $k\ge 4$, one may take
$$\del(\bfi)=\frac{2k+5}{(k^2+2k)+6}<\frac{2}{k-1/2},$$
whilst for large $k$, we may instead take
$$\del(\bfi)=\frac{2k+5}{2k(k+1)-\tfrac{4}{3}k+o(k)}=\frac{1}{k-\tfrac{13}{6}+o(1)}=
\frac{1}{k}+o(1).$$
In both situations, we conclude from Theorem \ref{theorem1.1} that there exists a set 
$\grB\subseteq [0,1)^{k-2}$ of full measure such that, when 
$(\alp_2,\alp_3,\ldots ,\alp_{k-1})\in \grB$, then for all real numbers $X$ sufficiently large 
in terms of $\eps$, $k$ and $\alp_2,\ldots ,\alp_{k-1}$, one has
$$\sup_{(\alp_1,\alp_k)\in [0,1)^2}|f_k(\bfalp;X)|\le X^{1/2+\del(\bfi)+\eps}.$$
This confirms both of the conclusions of Corollary \ref{corollary1.2}.\par

We turn next to Corollary \ref{corollary1.3}. Taking $\bfi=(i_1,\ldots ,i_t)$ and 
$u=\left[ \tfrac{1}{4}(k+1)^2\right]$, we find that the conclusion of Theorem 
\ref{theorem1.1} holds with
$$\del(\bfi)=\frac{t+1+2(i_1+\ldots +i_t)}{4[\tfrac{1}{4}(k+1)^2]+2t+2}
<\frac{k^2/\log k-t-1}{k^2+2k+2t+2}<\frac{1}{\log k}.$$
Consequently, there exists a set $\grB\subseteq [0,1)^{k-t}$ of full measure such that, 
when $\bfalp^\dagger\in \grB$, then for all real numbers $X$ sufficiently large in 
terms of $k$ and $\bfalp^\dagger$, one has 
${\displaystyle{\sup_{\bfalp^*\in [0,1)^t}|f_k(\bfalp;X)|\le X^{1/2+1/\log k}}}$. This 
confirms Corollary \ref{corollary1.3}. 

\section{Equidistribution of polynomials modulo one} We investigate the equidistribution of 
polynomial sequences by applying the Erd\H os-Tur\'an inequality (see 
\cite{ET1948a, ET1948b}). This entails estimating the exponential sum $f_k(h\bfalp;X)$ for 
$1\le h\le H$, with $H$ as large as is feasible. Suppose once more that $i_l$ $(1\le l\le t)$ 
are suffices with $1\le i_1<i_2<\ldots <i_t\le k$, with the conventions in the preamble to 
the statement of Theorem \ref{theorem1.1}. When $h\in \dbN$, we define a set 
generalising that defined in (\ref{2.1}) by putting
$$\grB_T^{(h)}(X)=\{\bfalp^\dagger \in [0,1)^{k-t}:\text{$|f_k(h\bfalp;X)|>T$ for some 
$\bfalp^*\in [0,1)^t$}\}.$$
Thus we have
\begin{equation}\label{3.1}
\grB_T^{(h)}(X)=\{ \bfalp^\dagger \in [0,1)^{k-t}:h\bfalp^\dagger \in \grB_T(X)
\mmod{1}\}.
\end{equation}
When $\lam,\mu\in \dbR$ and $\calA\subseteq \dbR$, denote by $\lam (\calA+\mu)$ the 
set $\{ \lam (\tet+\mu ):\tet\in \calA\}$. Then it follows from (\ref{3.1}) that
$$\grB_T^{(h)}(X)=\bigcup_{m=0}^{h-1}h^{-1}(\grB_T(X)+m),$$
and hence $\mu(\grB_T^{(h)}(X))=\mu(\grB_T(X))$ for $h\in \dbN$.\par

We next introduce the set $\grC_T(X,H)$ consisting of those points 
$\bfalp^\dagger \in [0,1)^{k-t}$ for which one has $|f_k(h\bfalp;X)|>T$ for some 
$\bfalp^*\in [0,1)^t$ and $h\in \dbN$ with $1\le h\le H$. Then we have
$$\grC_T(X,H)=\bigcup_{1\le h\le H}\grB_T^{(h)}(X),$$
so that
$$\mu(\grC_T(X,H))\le \sum_{1\le h\le H}\mu(\grB_T^{(h)}(X))
\le H\mu(\grB_T(X)).$$
We therefore deduce from Lemma \ref{lemma2.2} that when 
$1\le u\le \tfrac{1}{2}k(k+1)$ and ${\rm MC}_k(u)$ holds, then one has
\begin{equation}\label{3.2}
\mu(\grC_T(X,H))\ll HX^{u+t+\sig(\bfi)+\eps}T^{-2u-t}.
\end{equation}

\par We now make a choice for $T$ and $H$. Let $\tau$ be a positive number, and put
$$H(X)=X^{1/2-\nu(\bfi)-2\tau}\quad \text{and}\quad T(X)=X^{1/2+\nu(\bfi)+\tau},$$
where $\nu(\bfi)$ is defined as in (\ref{1.6}). Note that
\begin{equation}\label{3.3}
\tfrac{1}{2}+\nu(\bfi)=\frac{(2u+t+2)+(t+2+2\sig(\bfi))}{4u+2t+4}=
\frac{u+t+2+\sig(\bfi)}{2u+t+2}.
\end{equation}
We again consider a sequence of natural numbers $(X_n)_{n=1}^\infty$ satisfying the 
condition (\ref{2.4}), and then write $T_n=T(X_n)$ and $H_n=H(X_n)$.

\begin{lemma}\label{lemma3.1} Suppose that $1\le u\le \tfrac{1}{2}k(k+1)$ and assume 
${\rm MC}_k(u)$. Then for any sequence $(X_n)_{n=1}^\infty$ satisfying 
{\rm (\ref{2.4})}, one has
$$\mu\left( \bigcup_{n=1}^\infty \grC_{T_n}(X_n,H_n)\right) <\infty.$$
\end{lemma}

\begin{proof} In view of the relation (\ref{3.3}), it follows from (\ref{3.2}) that
\begin{align*}
\sum_{n=1}^\infty \mu (\grC_{T_n}(X_n,H_n))\ll \sum_{n=1}^\infty 
H_nX_n^{u+t+\sig(\bfi)+\eps}T_n^{-2u-t}\le \sum_{n=1}^\infty 
(H_nT_n^2/X_n^2)X_n^{\eps-2\tau}.
\end{align*}
The condition (\ref{2.4}) ensures that whenever $m\ge X_nT_n^{-1}$, then 
$X_{n+m}\ge 2X_n$. Thus, since $T_nH_n\le X_n$, one obtains
$$\mu\left( \bigcup_{n=1}^\infty \grC_{T_n}(X_n,H_n)\right) \ll \sum_{n=1}^\infty 
(T_n/X_n)X_n^{\eps-2\tau}\ll \sum_{j=0}^\infty (2^j)^{\eps-2\tau}<\infty.$$
This completes the proof of the lemma.
\end{proof}

Denote by $B_n(\bfalp^\dagger)$ the condition that $|f_k(h\bfalp;X_n)|>T_n$ for some 
$\bfalp^*\in [0,1)^t$ and $h\in \dbN$ with $1\le h\le H_n$. Then the definition of 
$\grC_T(X,H)$ implies that
$$\grC_{T_n}(X_n,H_n)=\{ \bfalp^\dagger\in [0,1)^{k-t}: B_n(\bfalp^\dagger)\}.$$
Put
$$\grC^*=\{\bfalp^\dagger\in [0,1)^{k-t}:\text{$B_n(\bfalp^\dagger)$ holds for infinitely 
many $n\in \dbN$}\}.$$
Then it follows from Lemma \ref{lemma3.1} via the Borel-Cantelli lemma that 
$\mu(\grC^*)=0$. Consequently, there is a set $\grC_0=[0,1)^{k-t}\setminus \grC^*$ of 
full measure having the property that, whenever $\bfalp^\dagger\in \grC_0$, then 
$B_n(\bfalp^\dagger)$ holds for at most finitely many $n\in \dbN$. The latter implies that 
$|f_k(h\bfalp;X_n)|\le T_n$ for all $\bfalp^*\in [0,1)^t$ and all $h\in \dbN$ with 
$1\le h\le H_n$, with the exception of at most finitely many $n\in \dbN$.

\par As in the corresponding treatment of \S2, the condition (\ref{2.4}) ensures that 
when $\bfalp^\dagger\in \grC_0$, then for all $X>0$ and $h\in \dbN$ with 
$h\le X^{1/2-\nu(\bfi)-2\tau}$, one has 
\begin{equation}\label{3.4}
\sup_{\bfalp^*\in [0,1)^t}|f_k(h\bfalp;X)|\le 3X^{1/2+\nu(\bfi)+\tau},
\end{equation}
except perhaps for certain numbers $X$ lying in a bounded interval $[0,X_0)$.\par

The estimate (\ref{3.4}) provides our basic input for an application of the Erd\H os-Tur\'an 
inequality, as decribed in \cite[Theorem 2.1]{Bak1986}. Suppose that $0\le a<b\le 1$. Also, write 
$x_n=\alp_kn^k+\ldots +\alp_1n$ and put $H=X^{1/2-\nu(\bfi)-2\tau}$. Then
\begin{align*}
\biggl| \sum_{\substack{1\le n\le X\\ x_n\in [a,b]\mmod{1}}}1-X(b-a)\biggr| &\le 
\frac{X}{H+1}+3\sum_{1\le h\le H}h^{-1}\biggl|\sum_{1\le n\le X}e(hx_n)\biggr| \\
&=\frac{X}{H+1}+3\sum_{1\le h\le H}h^{-1}|f_k(h\bfalp;X)|.
\end{align*}
Consequently, when $\bfalp^\dagger\in \grC_0$, one finds from (\ref{3.4}) that
$$|Z_{a,b}(\bfalp;X)-X(b-a)|\le XH^{-1}+
9\sum_{1\le h\le H}h^{-1}X^{1/2+\nu(\bfi)+\tau}\ll X^{1/2+\nu(\bfi)+2\tau}.$$
Since $\tau>0$ may be taken arbitarily small, Theorem \ref{theorem1.4} now follows.\vskip.1cm

The proof of Corollary \ref{corollary1.5} follows on taking $\bfi=(k)$ and $u=\left[ 
\tfrac{1}{4}(k+1)^2\right]$, so that the conclusion of Theorem \ref{theorem1.4} holds 
with
$$\nu(\bfi)=\frac{3+2k}{4u+6}\le \frac{2k+3}{k^2+2k+6}<\frac{2}{k}.$$
Then we conclude that there exists a set $\grB^*\subseteq [0,1)^{k-1}$ of full measure 
with the property that, whenever $(\alp_1,\ldots ,\alp_{k-1})\in \grB^*$, then for all 
$N\in \dbN$ sufficiently large in terms of $k$ and $\alp_1,\ldots ,\alp_{k-1}$, one has
$$|Z_{a,b}(\bfalp;N)-(b-a)N|\le N^{1/2+2/k}\quad (0\le a<b\le 1).$$
This completes the proof of Corollary \ref{corollary1.5}.

\bibliographystyle{amsbracket}
\providecommand{\bysame}{\leavevmode\hbox to3em{\hrulefill}\thinspace}

\end{document}